\documentclass{elsarticle}

\usepackage{amsthm, amsmath, amssymb, amscd,  amsfonts, amstext,
  amsbsy, mathrsfs}

\newcommand{\pow}{\ensuremath{\mathscr{P}}}
\newcommand{\RR}{\ensuremath{\mathbb{R}}}
\newcommand{\Bai}{\ensuremath{{}^\omega \omega}}
\newcommand{\Can}{\ensuremath{{}^\omega 2}}
\newcommand{\AD}{\ensuremath{{\rm \mathsf{AD}}}}
\newcommand{\ZF}{\ensuremath{{\rm \mathsf{ZF}}}}
\newcommand{\DC}{\ensuremath{{\rm \mathsf{DC}}}}
\newcommand{\ACOR}{\ensuremath{{\rm \mathsf{AC}_\omega(\mathbb{R})}}}
\newcommand{\SLOL}{\ensuremath{{\rm \mathsf{SLO^L}}}}
\newcommand{\SLO}{\ensuremath{{\rm \mathsf{SLO}}}}
\newcommand{\ax}[1]{\ensuremath{{\rm \mathsf{#1}}}}
\newcommand{\imp}{\Rightarrow}
\newcommand{\fhi}{\varphi}
\newcommand{\bSigma}{\mathbf{\Sigma}}
\newcommand{\bPi}{\mathbf{\Pi}}
\newcommand{\bGamma}{\ensuremath{\mathbf{\Gamma}}}
\newcommand{\bDelta}{\mathbf{\Delta}}
\newcommand{\bLambda}{\mathbf{\Lambda}}
\newcommand{\bN}{\mathbf{N}}
\newcommand{\bdelta}{\boldsymbol{\delta}}
\newcommand{\F}{\mathcal{F}}
\newcommand{\G}{\mathcal{G}}
\newcommand{\Lip}{\ensuremath{{\mathsf{Lip}}}}
\newcommand{\Bor}{\ensuremath{{\mathsf{Bor}}}}
\renewcommand{\L}{\ensuremath{{\mathsf{L}}}}
\newcommand{\restr}[2]{#1 \restriction #2}
\newcommand{\seq}[2]{\langle #1 \mid  #2 \rangle}
\newcommand{\onto}{\twoheadrightarrow}

\newtheorem{theorem}{Theorem}[section]
\newtheorem{lemma}[theorem]{Lemma}
\newtheorem{proposition}[theorem]{Proposition}
\theoremstyle{definition}
\newtheorem{defin}{Definition}

\begin{document}

\title{Beyond Borel-amenability: scales and superamenable reducibilities}
\author{L.\ Motto Ros\fnref{thanks}\corref{tel}}
\date{\today}
\ead{luca.mottoros@libero.it}
\address{Kurt G\"odel Research Center for Mathematical Logic, University of Vienna, W\"ahringer Stra{\ss}e 25,
 A-1090 Vienna,
Austria}
\fntext[thanks]{The author would like to thank the FWF (Austrian Research Fund) for generously supporting this research throught Project number P 19898-N18.}
\cortext[tel]{Tel.: +43 1 427750508; fax: +43 1 427750599.}
\begin{keyword}
Determinacy \sep Wadge hierarchy

\MSC 03E15 \sep 03E60
\end{keyword}

\begin{abstract}
We analyze the degree-structure induced by large reducibilities under the Axiom of Determinacy.
This generalizes the analysis of Borel reducibilities given in \cite{andrettamartin},  \cite{mottorosborelamenability} and  \cite{mottorosbaire} e.g.\ to the projective levels.
\end{abstract}

\maketitle

\section{Introduction}
Given a set of functions $\F$ from\footnote{As usual in Descriptive Set Theory we will always identify $\RR$ with the Baire space $\Bai$.} $\RR$ into itself (also called a \emph{reducibility}), we say that $A,B \subseteq \RR$ are \emph{$\F$-equivalent} if each of them is the $\F$-preimage of the other one, and call \emph{$\F$-degree} of $A$  the collection of all sets $\F$-equivalent to $A$: our main goal is to study the structure of the $\F$-degrees for various $\F$.
Building on the work of Andretta and Martin in \cite{andrettamartin} (where the case when $\F$ is the set of all Borel functions was considered), in \cite{mottorosborelamenability} and \cite{mottorosbaire} we have investigated various reducibility notions in the Borel context, but it is clear that there are also some natural sets of functions (such as projective functions) that can be used as reductions and which are strictly larger than the set of Borel functions. In this paper we will prove that, assuming $\AD+\DC$, structural results similar to those for the Borel context can be proved for larger and larger pointclasses. In particular, we will determine the degree-structure induced by the collection $\F_\bGamma$ of all \emph{$\bGamma$-functions} (i.e.\ of those functions with the property that the preimage of a set in $\bGamma$ is still in $\bGamma$) in case $\bGamma$ is a boldface pointclass which is closed under projections, countably intersections and unions, and which has the scale and the uniformization property (under $\AD$ these $\bGamma$'s coincide with the so-called \emph{tractable pointclasses} --- see Section \ref{sectionexistentialpointclasses}).

The existence of such pointclasses is strictly related to the axioms
one is willing to accept. For example, in $\ZF+\ACOR$ the only known tractable
pointclass is $\bSigma^1_2$, but in general the stronger the axioms one is
willing to adopt,
the greater number of tractable pointclasses one gets
(see \cite{mosch}, \cite{martinsteel} and \cite{woodin} for the results quoted below):
\begin{enumerate}
\item 
$\ax{Det}(\bDelta^1_{2n})$ implies that there are at least $n+1$ tractable pointclasses, namely
$\bSigma^1_2, \dotsc, \bSigma^1_{2n+2}$. In particular, Projective Determinacy $\ax{Det}(\bigcup_n\bDelta^1_n)$ implies that
each pointclass $\bSigma^1_{2n+2}$ is tractable.
A similar result holds for the even levels of the $\sigma$-projective pointclass;

\item if $\lambda$ is an ordinal of uncountable 
cofinality and 
 $\seq{\bGamma_\xi}{\xi<\lambda}$ is a
chain of tractable pointclasses (i.e.\ $\bGamma_\xi \subsetneq
\bGamma_{\xi'}$ for $\xi < \xi'$), then the pointclass
$\bGamma = \bigcup_{\xi<\lambda} \bGamma_\xi$ is
tractable as well.
In particular, $\sigma$-Projective Determinacy $\ax{Det}(\bigcup_{\xi<\omega_1}\bDelta^1_\xi)$ implies that  the pointclass of all  $\sigma$-projective sets is tractable (while the pointclass of all projective sets is \emph{not} tractable);

\item Hyperprojective Determinacy $\ax{Det}(\mathbf{HYP})$ implies that the collection of all inductive sets is tractable;

\item if $\delta$ is limit
 of Woodin cardinals then $\bGamma^H_{<\delta}$, the collection of all $\xi$-weakly
homogeneously Suslin sets (for any $\xi <\delta$), is a
tractable
pointclass;

\item 
assuming $\AD + \ax{V =  L(\RR)}$, the
pointclass $\bSigma^2_1$ is scaled (hence it has also the uniformization
property by closure under coprojections), but 
if $\ax{V= L(\RR)}$ and there is no wellordering of the reals
then there is a $\bPi^2_1$ subset of $\RR^2$ that can not be uniformized (by
\emph{any} set in $\RR^2$): thus, if we assume $\AD + \ax{V = L(\RR)}$, we get that $\bSigma^2_1$
is the \emph{maximal}
tractable pointclass;

\item in contrast with the previous point,
Woodin has shown that $\AD_\RR$ implies that \emph{every} set
of reals has a scale, and this in turn implies, by previous work
of Martin, that there are nonselfdual scaled pointclasses with reasonable closure properties which lie arbitrarily high in
the Wadge ordering: thus, in particular, under $\AD_\RR$ there are
tractable pointclasses of arbitrarily high complexity.
\end{enumerate}

All these examples show that our arguments allow to determine the degree-structures induced by larger and larger sets of reductions (assuming corresponding determinacy axioms). Nevertheless we have to point out that at the moment we are able to deal e.g.\ with $\bSigma^1_{2n}$-reductions but not with $\bSigma^1_{2n+1}$-reductions (for $n>0$). This asymmetry arises from the zig-zag pattern of the regularity properties given by Moschovakis' Periodicity Theorems, and reflects a phenomenon which is quite common in the $\AD$ context: for instance, in \cite{vanwesepwadgedegrees} it was shown that the order type of the $\bDelta^1_{2n}$ degrees is exactly $\bdelta^1_{2n+1}$, but no exact evaluation of the order type of $\bDelta^1_{2n+1}$ has been given so far (apart from the inequalities $\bdelta^1_{2n+1} < \text{order type of }\bDelta^1_{2n+1} < \bdelta^1_{2n+2}$). 

Another important feature of large reductions is that to have our structural results we always need the full $\AD$, as we have to use the Moschovakis' Coding Lemmas: this should be contrasted with the Borel case, in which the determinacy axioms were used only in a local way. We finish this introduction by aknowledging our debt to A.~Andretta and D.~A.~Martin for their \cite{andrettamartin} and for the simple but crucial suggestion of using scales (instead of changes of topology) in the present setup.

\section{Basic facts and superamenability}\label{sectionsuperamenability}

We will firstly recall some definitions and basic facts for the reader's convenience.
For all undefined symbols, terminology, and for the proofs omitted here we refer the reader to \cite{kechris}, \cite{mosch} and \cite{mottorosborelamenability}.
If $\bGamma \subseteq \pow(\RR)$ is any boldface pointclass, we say
that the surjection $\fhi \colon P \onto \lambda$ is a \emph{$\bGamma$-norm} if there
are
relations $\leq^\fhi_\bGamma$ and $\leq^\fhi_{\breve{\bGamma}}$ in $\bGamma$ and
$\breve{\bGamma}$, respectively, such that for every $y \in P$ and every $x \in
\RR$
\[x \in P \wedge \fhi(x) \leq \fhi(y) \iff x \leq^\fhi_\bGamma y \iff
x \leq^\fhi_{\breve{\bGamma}} y.\]
To each norm $\fhi$ we can associate the prewellordering (i.e.\ the transitive, reflexive, connected and well-founded relation) $\leq^\fhi$ defined by $x \leq^\fhi y \iff \fhi(x) \leq \fhi(y)$ 
(for every $x,y \in P$).
A pointclass $\bGamma$ is said to be \emph{normed} if every $P \in \bGamma$ admits a $\bGamma$-norm.
In this case, if $\bGamma$ is a boldface pointclass closed under finite
intersections and unions, then $\bGamma$ has the reduction property while
 $\breve{\bGamma}$ has the separation property, and if $\fhi$ is a
$\bGamma$-norm on a set $D \in \bDelta_\bGamma$ then $\leq^\fhi$ is in
$\bDelta_\bGamma$.
Moreover, we can define
\[ \bdelta_\bGamma = \sup\{ \xi \mid \xi \text{ is the length of a
prewellordering of } \RR \text{ which is in }
\bDelta_\bGamma\}\]
(clearly $\bdelta_\bGamma < (2^{\aleph_0})^+$ for every $\bGamma \subsetneq
\pow(\RR)$).
If $\bGamma$ is closed under coprojections,
countable intersections and countable unions, then there is a regular
$\bGamma$-norm with length $\bdelta_\bGamma$: this implies that for every pointclass $\bLambda$
\begin{equation}\label{eq1}
\bGamma \cup \breve{\bGamma}
\subseteq \bLambda \imp \bdelta_\bGamma < \bdelta_\bLambda.
\end{equation}

A \emph{$\bGamma$-scale} on $P \subseteq \RR$ is  a sequence $\vec{\fhi} = \seq{\fhi_n}{n \in \omega}$ of norms on $P$ such that
\begin{enumerate}
 \item if $x_0, x_1, \dotsc \in P$, $\lim_i x_i = x$ for some $x$, and for each $n$ we have $\lim_i \fhi_n(x_i) = \lambda_n$ for some ordinal $\lambda_n$, then $x \in P$ and $\fhi_n(x) \leq \lambda_n$ for each $n$;

\item there are relations $S_\bGamma(n,x,y)$ and $S_{\breve{\bGamma}}(n,x,y)$ in
$\bGamma$ and $\breve{\bGamma}$ respectively such that for every $y \in P$,
every $n \in \omega$ and every $x \in \RR$
\[ x \in P \wedge \fhi_n(x) \leq \fhi_n(y) \iff S_\bGamma(n,x,y) \iff
S_{\breve{\bGamma}}(n,x,y).\]
\end{enumerate}

If every set in $\bGamma$ admits a $\bGamma$-scale we say that the
 pointclass $\bGamma$ is \emph{scaled}, and in this case if 
$\bGamma$ is closed under finite intersections and unions we can also require that on each $P \in \bGamma$ there is a $\bGamma$-scale such that
\begin{equation}\label{verygood}
\tag{$\star$} \begin{split}
& \text{if }x_0, x_1, \dotsc \in P\text{ and for each }n\text{ we have }\lim\nolimits_i \fhi_n(x_i) = \lambda_n\text{ for some }\lambda_n\text{, then}\\ & \text{\emph{there exists some $x \in P$ for which $\lim\nolimits_i x_i = x$}}.
\end{split}
\end{equation}

If $\bGamma$ is scaled and closed under coprojections then $\bGamma$ has the
\emph{uniformization
property}, i.e.\ for every $P \subseteq \RR
\times \RR$ which is in $\bGamma$ there is some $P^* \subseteq P$
such that  for every $x$ in the projection of $P$ there is a \emph{unique}
$y \in \RR$ that satisfies $(x,y) \in P^*$ (and in this case we will say that
$P^*$ \emph{uniformizes} $P$). The same is true also for the pointclass
$\exists \bGamma = \{A \subseteq \RR \mid A \text{ is the projection of a
set in }\bGamma\}$,
that is $\exists \bGamma$ is scaled and has the uniformization property.

Finally, we want to recall some
results which are consequences of the full $\AD$.

\begin{lemma}[First Coding Lemma]\label{codinglemma}
  Assume $\AD$ and let $<$ be a strict well-founded relation on some $S
\subseteq \RR$ with rank function $\rho\colon S \onto \lambda$.
Moreover, let
$\bGamma \supseteq \bDelta^0_1$
be a pointclass closed under projections, countable unions and
countable intersections, and assume that ${<} \in {\bGamma}$.
Then for every function
$ f \colon \lambda \to \pow(\RR)$
there is a \emph{choice set} $C \in \bGamma$, that is a set $C \subseteq \RR \times \RR$ such that
\begin{enumerate}[i)]
\item $(x,y) \in C \imp x \in S \wedge
y \in f(\rho(x))$,
\item $f(\xi) \neq \emptyset \imp \exists x  \exists y  (\rho(x) =
\xi \wedge 
(x,y) \in C)$.
\end{enumerate}
\end{lemma}

Note that our
formulation of Lemma \ref{codinglemma} is slightly
different from the original  one (due to Moschovakis): nevertheless, one can easily
check that our statement is a  particular case (and hence
a consequence) of the Moschovakis' one.
Using a similar reformulation of the Second Coding Lemma, we get that if $\bGamma \supseteq \bDelta^0_1$ is closed under projections,
countable unions and countable intersections, then $\AD$ implies that
$\bdelta_\bGamma$ is a cardinal of uncountable cofinality.
Thus, in particular, if $\vec{\fhi} = \seq{\fhi_n}{n \in \omega}$
is a scale on a set $D \in
\bDelta_\bGamma$ then there is  $\lambda<\bdelta_\bGamma$ such that $\fhi_n
\colon D \to \lambda$ for every $n \in \omega$.
Moreover we have that $\bigcup_{\xi<\lambda}A_\xi \in \bGamma$ for every $\lambda < \bdelta_\bGamma$ and
every family
$\{A_\xi \mid \xi<\lambda\} \subseteq \bGamma$.

If we assume also $\DC$ we get that for every
$n \in \omega$
\begin{equation}\label{eq2}
 A \in \bSigma^1_{2n+2} \iff A \text{ is the union of } \bdelta^1_{2n+1}
\text{-many sets in }\bDelta^1_{2n+1},
 \end{equation}
where for every $n$ we put $\bdelta^1_n = \bdelta_{\bSigma^1_n} = \bdelta_{\bPi^1_n}$.
On the other hand, 
assuming  $\AD+\DC$ we have that
$ \bDelta^1_{2n+1} = \mathbf{B}_{2n+1}$,
where $\mathbf{B}_{2n+1}$ is the least boldface pointclass which contains
all the open sets and is closed under complementation and unions of length less
then $\bdelta^1_{2n+1}$ (i.e.\ it is the least $\bdelta^1_{2n+1}$-complete algebra of sets
which contains all the open sets). Thus under $\AD+\DC$
the projective sets are completely
determined by the \emph{projective ordinals}
 $\bdelta^1_n$ and the operations of
complementation and ``well-ordered''
union. All these facts together allow us to prove the following simple lemma.

\begin{lemma}\label{lemmanew}
  Assume $\AD+\DC$ and let $n \neq 0$ be an even number. Then $D \in
  \bDelta^1_n$ if and only if  there is a $\bDelta^1_{n-1}$-partition
  $\seq{D_\xi}{\xi<\bdelta^1_{n-1}}$ of $\RR$ such that $D =
  \bigcup_{\xi \in S} D_\xi$ for some $S \subseteq
  \bdelta^1_{n-1}$. 
\end{lemma}

\begin{proof}
 By \eqref{eq2},  $A \in \bSigma^1_n$ if and only if
  $A$ is the union of a family $\mathcal{B}=\{B_\xi \subseteq \RR \mid
  \xi<\bdelta^1_{n-1}\} \subseteq \bDelta^1_{n-1}$. Since $\bDelta^1_{n-1} = \mathbf{B}_{n-1}$,
  we can refine $\mathcal{B}$ to a pairwise disjoint family with the same properties by defining
  $B'_\xi = B_\xi
  \setminus \bigcup_{\mu<\xi} B_\mu$ for every
  $\xi<\bdelta^1_{n-1}$. Applying
this argument to both $D$ and $\neg D$ we get the result.
The converse is obvious since $\bSigma^1_n$ is closed under well-ordered
unions of length smaller than $\bdelta^1_n$.
\end{proof}

A set of functions $\F$ from $\RR$ into itself is called \emph{set of reductions} if it is closed under composition, contains $\L$ (= the collection of all Lipschitz functions with constant $\leq 1$), and admits a surjection $j \colon \RR \onto \F$. Given such an $\F$, we put $A \leq_\F B$ if and only if $A = f^{-1}(B)$ for some $f \in \F$. Since $\leq_\F$ is a preorder, we can consider the associated equivalence relation $\equiv_\F$ and the corresponding $\F$-degrees $[A]_\F = \{ B \subseteq \RR \mid A \equiv_\F B \}$. Our main goal is to determine the structure of the $\F$-degrees with respect to the preorder induced on them by $\leq_\F$. Notice that under $\AD$ we have the Semi-linear Ordering Principle for $\F$
\begin{equation}
 \tag{$\SLO^\F$} \forall A,B \subseteq \RR( {A \leq_\F B} \vee {B \leq_\F \neg A}).
\end{equation}
As already pointed out in \cite{mottorosborelamenability}, the arguments used to determine the degree-structures induced by  Borel reducibilities (namely changes of topology) cannot be applied outside the Borel context without loosing the crucial property that the new topology is still Polish. Moreover, one can see that the dichotomy countable/uncountable is inadequate when dealing with large reductions, and new ordinals must be involved. The natural choice is to consider the  
 \emph{characteristic
  ordinal} of $\F$
\[ \bdelta_\F = \sup \{ \xi \mid \xi \text{ is the length of a
  prewellordering of } \RR \text{ which is in } \Delta_\F\},\]
where $\Delta_\F = \{ A \subseteq \RR \mid A \leq_\F \bN_{\langle 0 \rangle}\}$ is the \emph{characteristic set} of $\F$.

It is clear that if $\F \subseteq \G$ (or even just if $\Delta_\F
\subseteq \Delta_\G$) then $\bdelta_\F \leq \bdelta_\G$. In general the
converse is not true --- see the observation below.
Using this ordinal we can give the following definition.

\begin{defin}\label{defsuperamenable}
  A set of reductions $\F$ is called \emph{superamenable} if
  \begin{enumerate}[i)]
  \item $\Lip \subseteq \F$, where $\Lip$ is the set of all Lipschitz functions (irrespective of their constant);
  \item for every $\eta < \bdelta_\F$, every $\Delta_\F$-partition\footnote{A \emph{$\Gamma$-partition of a set $A \subseteq \RR$}
  is simply a sequence
$\seq{A_\xi}{\xi<\lambda}$ of pairwise disjoint sets in
$\Gamma$ such that $\lambda$ is
an ordinal and $A = \bigcup_{\xi<\lambda} A_\xi$.}
    $\seq{D_\xi}{\xi < \eta}$ of $\RR$ and every sequence of functions
    $\seq{f_\xi}{\xi < \eta}$ we have that
\[ f = \bigcup\nolimits_{\xi < \eta} (\restr{f_\xi}{D_\xi}) \in \F.\]
  \end{enumerate}
\end{defin}

Superamenability is clearly a natural extension of Borel-amenability as presented in \cite{mottorosborelamenability}, since any set of reductions $\F \subseteq \Bor$ is
Borel-amenable if and
only if it
is superamenable. (This is because 
$\bdelta_\Lip = \bdelta_\Bor
=\bdelta^1_1 = \omega_1$.
To see this, it is clearly enough to show that $\bdelta_\Lip \geq
\omega_1$: let $\alpha < \omega_1$ and $z \in WO_\alpha = \{ w \in \RR \mid
w \text{ codes a wellordering $\leq_z$ of $\omega$ of length } \alpha\}$. Then for every
$x,y \in \RR$ put
\[ x \leq y \iff z(\langle x(0),y(0) \rangle)=1 \iff x(0) \leq_z
y(0).\]
It is clear that this is a prewellordering on $\RR$ of length $\alpha$,
and one can easily check that its image under the
canonical homeomorphism between $\RR^2$ and $\RR$ is in $[\bN_{\langle
0,0 \rangle}]_\L \subseteq \Delta_\Lip$. Thus if $\Lip \subseteq \F
\subseteq \Bor$ then $\bdelta_\F = \omega_1$: in particular, all the
Borel-amenable sets of reductions $\F$ give rise to the same
characteristic ordinal
$\bdelta_\F = \omega_1$, and from this easily follows that  in this case the two definitions coincide.)\\

If we assume $\AD+ \DC$, a particular place among the superamenable
sets of reductions which are  subsets of the \emph{projective
functions}
is occupied by the $\bDelta^1_{2n+2}$-functions
--- see also the next section.

\begin{proposition}
  Assume $\AD+\DC$ and let $\F$ be a superamenable set of
  reductions such that $\Delta_\F$ is a proper subset of the collection of
  the projective sets. Let $n$ be the smallest natural number such
  that  $\Delta_\F
  \subseteq \bDelta^1_n$. If $n$ is even
 then $\Delta_\F =
  \bDelta^1_n$ (and hence also $\bdelta_\F = \bdelta^1_n$).
The same conclusion holds also if $n$ is the smallest natural number
such that $\bdelta_\F \leq \bdelta^1_n$ (and $n$ is even again).
\end{proposition}

\begin{proof}
First we prove that $\bdelta^1_{n-1} < \bdelta_\F \leq
\bdelta^1_n$. Since $\Delta_\F$ is closed under $\L$-preimages, by $\SLO^\L$ (which is a consequence of $\AD$) either $\Delta_\F \subseteq \bDelta^1_{n-1}$ or else $\bDelta^1_{n-1} \subsetneq \Delta_\F$. The minimality of $n$ implies the second possibility, and since  ${\bDelta^1_{n-1} \subsetneq \Delta_\F} \imp {{\bSigma^1_{n-1} \cup \bPi^1_{n-1}} \subseteq \Delta_\F}$ by $\SLO^\L$ again, applying \eqref{eq1} with $\bGamma = \bPi^1_{n-1}$ we get $\bdelta^1_{n-1} < \bdelta_\F$. Finally, $\bdelta_\F \leq \bdelta^1_n$ by the choice of $n$.
 
Now let $D \in \bDelta^1_n$ and
$\seq{D_\xi}{\xi<\bdelta^1_{n-1}}$ be a $\bDelta^1_{n-1}$-partition of
$\RR$ such that $D = \bigcup_{\xi \in
  S} D_\xi$ for some $S \subseteq \bdelta^1_{n-1}$ (such a partition exists by Lemma \ref{lemmanew}). Moreover, let
$f_i$ be the constant function with value $\vec{i} = \langle i,i,i, \dots \rangle$,  and 
for every $\xi < \bdelta^1_{n-1}$ put
$f_\xi =
f_0$ if $\xi \in S$ and
$f_\xi =f_1$ otherwise. It is clear that $f_0, f_1 \in \L
\subseteq \F$ and that since $\bdelta^1_{n-1} < \bdelta_\F$
\[ f = \bigcup_{\xi<\bdelta^1_{n-1}}(\restr{f_\xi}{D_\xi}) \in \F\]
by superamenability. But  $f$ reduces $D$ to
$\bN_{\langle 0 \rangle}$, hence $D \in \Delta_\F$.
\end{proof}

Recall that by Theorem 3.1 of \cite{mottorosborelamenability}, the structure of the
$\F$-degrees is completely determined whenever we can establish what
happens at limit levels
(of uncountable cofinality) and after a selfdual degree. Moreover, we have
that  Lemma 4.4 of \cite{mottorosborelamenability} holds in our new context (hence, in particular,
$D \cap A \leq_\F A$ for every $D \in \Delta_\F$ and every $A \neq \RR$),
but the definition of the decomposition property given in that paper must be adapted
to the new setup.

\begin{defin}
  Let $\F$ be a superamenable set of reductions.
A set $A \subseteq \RR$ has the \emph{decomposition property} with respect
to $\F$ if there is some $\eta
< \bdelta_\F$ and a $\Delta_\F$-partition
$\seq{D_\xi}{\xi<\eta}$ of $\RR$
such that $D_\xi \cap A <_\F A$ for every $\xi<\eta$.

The set of reductions $\F$ has the \emph{decomposition property}
(\textbf{DP} for short) if every
$\F$-selfdual $A \subseteq \RR$ such that $A \notin \Delta_\F$ has the
decomposition
property with respect to $\F$.
\end{defin}

Note that the new definition of the 
decomposition property
is coherent (i.e.\ coincide) with the original one whenever $\F \subseteq \Bor$, as 
this implies $\bdelta_\F
= \omega_1$.

\section{Tractable pointclasses and large reducibilities}
\label{sectionexistentialpointclasses}

We call \emph{existential pointclass} any
  boldface pointclass $\bDelta^0_1 \subseteq \bGamma \neq \pow(\RR)$ which is closed under projections,
countable
  unions and countable intersections.
For instance, the projective pointclasses $\bSigma^1_n$ and the
collections $\mathbf{S}(\kappa)$ of all $\kappa$-Suslin sets
(where $\kappa$ is some infinite cardinal) are  existential pointclasses.

Moreover we will call 
\emph{tractable pointclasses} those existential pointclasses $\bGamma$'s 
which have the uniformization property and such that either $\bGamma$ or
$\breve{\bGamma}$ is scaled. Notice that not all the existential pointclasses are tractable, as e.g.\ $\bSigma^1_{2n+1}$ does not have the uniformization property.
Note also that if $\bGamma$ has a universal set then $\bGamma$ 
is tractable if
and only if $\bGamma$ is a scaled existential pointclass with the
uniformization property. (Assume towards a contradiction
that $\bGamma$ is an existential pointclass with the
uniformization property and that $\breve{\bGamma}$ is scaled: since $\breve{\bGamma}$ is also
closed under coprojections, we would have that $\breve{\bGamma}$ has
the uniformization property as well,
and this would in turn imply that both $\bGamma$ and $\breve{\bGamma}$ have the
reduction property. But this contradicts a standard fact in Descriptive Set Theory, see e.g.\ Proposition 22.15 in \cite{kechris}.) In particular, this equivalence is true under $\AD$ (as $\SLOL$ implies that any nonselfdual boldface pointclass has a universal set).

\begin{proposition}\label{propdeltagammafunctions}
  Let $\bGamma$ be an existential pointclass, and let $f\colon \RR \to \RR$
  be any function. Then the following are equivalent:
  \begin{enumerate}[i)]
  \item $f$ is a $\bGamma$-function (equivalently, a $\breve{\bGamma}$-function);
\item $f$ is a $\bDelta_\bGamma$-function;
\item $f$ is $\bDelta_\bGamma$-measurable (equivalently, $\bGamma$-measurable);
\item ${\rm graph}(f) \in \bDelta_\bGamma$, where ${\rm graph}(f) = \{ (x,y) \in \RR^2 \mid f(x) = y \}$;
\item ${\rm graph}(f) \in \bGamma$.
  \end{enumerate}
\end{proposition}

\begin{proof}
  It is not hard to see that \textit{i)}  implies \textit{ii)}, and since $\bDelta_\bGamma$ is
  closed
under countable unions and intersections, we have that \textit{ii)}
implies
\textit{iii)}.
Moreover, \textit{iii)} implies \textit{iv)} since $\bDelta_\bGamma$ is closed under countable intersections and
\[ (x,y) \in {\rm graph}(f) \iff \forall n (x \in
f^{-1}(\bN_{\restr{y}{n}})). \]
Clearly
\textit{iv)} implies \textit{v)} and, finally, \textit{v)} implies
\textit{i)} since if $A \in \bGamma$ then
\[ f^{-1}(A) = \{ x \in \RR \mid \exists y((x,y) \in {\rm graph}(f)
\wedge y \in A) \} \in \bGamma \]
by closure of $\bGamma$ under projections and finite intersections.
\end{proof}

Given an existential pointclass $\bGamma$, we can define the set of functions
\[ \F_\bGamma = \{ f \colon \RR \to \RR \mid f \text{ is a }
\bDelta_\bGamma\text{-function}\}\]
(equivalently, $\F_\bGamma$ is the collection of all functions  which
satisfy any of the conditions 
in Proposition \ref{propdeltagammafunctions}), and
it is immediate to check 
that
$\Delta_{\F_\bGamma} = \bDelta_\bGamma$ and $\bdelta_{\F_\bGamma} = \bdelta_\bGamma$ is a cardinal of uncountable cofinality  --- see Section \ref{sectionsuperamenability}. A set of reductions $\F$ will be called \emph{tractable} if 
$\F = \F_\bGamma$ for some tractable pointclass $\bGamma$.

Now assume $\AD$ and let $\bGamma$ be an existential pointclass. Using the fact that $\bGamma \neq \pow(\RR)$,
by Remark 3.2 of \cite{mottorosborelamenability} we have
that
there  is a
surjection of $\RR$ onto $\F_\bGamma$, and thus $\F_\bGamma$ is automatically a  set of
reductions since it is trivially closed under composition. Moreover, $\Lip \subseteq \Bor \subseteq 
\F_\bGamma$ and, using the fact that
$\bGamma$ is closed under unions of length less than $\bdelta_\bGamma$, we have
that $\F_\bGamma$ is a superamenable set of reductions: in fact, if $A \in
\bGamma$ and $f = \bigcup_{\xi<\lambda}(\restr{f_\xi}{D_\xi})$ (with $\lambda <
\bdelta_\bGamma$, $f_\xi \in \F_\bGamma$, and $\seq{D_\xi}{\xi<\lambda}$ a
$\bDelta_\bGamma$-partition of $\RR$) we have
\[ f^{-1}(A) = \bigcup\nolimits_{\xi<\lambda}(f_\xi^{-1}(A) \cap D_\xi) \in \bGamma,\]
hence $f \in \F_\bGamma$.
In particular, the set of all
$\bDelta^1_n$-functions (for each $n$) is superamenable.

We will now try to determine, under $\AD+\DC$, the structure of degrees induced
by $\F_\bGamma$. For the sake of simplicity, we will sistematically use the
symbol $\bGamma$ instead of $\F_\bGamma$ in all the notations
related to reductions: for example, we will write $\leq_\bGamma$,
$[A]_\bGamma$, $\SLO^\bGamma$ instead of
$\leq_{\F_\bGamma}$, $[A]_{\F_\bGamma}$, $\SLO^{\F_\bGamma}$,
 and so on.

The first step is to prove that $\F_\bGamma$ has the decomposition
property,
but
to have this result we must assume that either $\bGamma$ or $\breve{\bGamma}$
has the scale property. In both cases, we have that every $D \in
\bDelta_\bGamma$ admits a $\bGamma$-scale $\vec{\psi} =
\seq{\psi^D_n}{n
\in \omega}$ on $D$ with the property \eqref{verygood} and such that $\psi^D_n\colon D \onto \eta_D$ (for some $\eta_D <
\bdelta_\bGamma$). Similarly, if $f \in \F_\bGamma$ then there is a
$\bGamma$-scale $\vec{\psi} = \seq{\psi^f_n}{n \in \omega}$ on ${\rm
  graph}(f)$ and an ordinal
$\eta_f < \bdelta_\bGamma$ such that $\vec{\psi}$ has the property \eqref{verygood} and $\psi^f_n \colon {\rm graph}(f) \onto \eta_f$
for every $n \in \omega$.

\begin{theorem}[\AD]\label{theorDP2}
  Let $\bGamma$ be an existential pointclass such that either $\bGamma$ or
$\breve{\bGamma}$ is scaled. Then every $\F_\bGamma$-selfdual $A \subseteq \RR$
such that $A \notin \Delta_{\F_\bGamma} = \bDelta_\bGamma$ has the decomposition property with
respect to $\F$.
\end{theorem}

\begin{proof}
  Towards a contradiction with $\AD$, assume that for every $\eta<\bdelta_\bGamma$
  and every $\bDelta_\bGamma$-partition $\seq{D_\xi}{\xi<\eta}$ of
  $\RR$ there is some $\xi_0 < \eta$ such that $A \cap D_{\xi_0}
  \equiv_\bGamma A$: we will construct a \emph{flip-set}, that is a subset $F$ of the Cantor space $\Can$ with the property that $z \in F \iff w \notin F$ whenever $z,w \in \Can$ and $\exists ! n (z(n) \neq w(n))$. Since every flip-set can not have the Baire property, this will give the desired contradiction.

The ideas involved in the present proof are not far from those used for Theorem 5.3
 of \cite{mottorosborelamenability}, but in this case we will use $\bGamma$-scales instead of changes
of topology. Let us say that $A$ is \emph{not decomposable in $D \in \bDelta_\bGamma$} if there is no $\eta < \bdelta_\bGamma$ and no $\bDelta_\bGamma$-partition $\seq{D_\xi}{\xi<\eta}$ of $D$ such that $A \cap D_\xi <_\bGamma A$ for each $\xi < \eta$.
Arguing as in the original proof, one can prove that if $A$ is not decomposable in some
$D \in \bDelta_\bGamma$ then
there is some $f \in \F_\bGamma$ such that ${\rm range}(f) \subseteq
D$ and
\[ \forall x \in D(x \in A \cap D \iff f(x) \in \neg A \cap D).\]

We will construct a countable sequence of nonempty $\bDelta_\bGamma$-sets
\[ \dotsc \subseteq D_2 \subseteq D_1 \subseteq D_0 = \RR, \]
a sequence of $\bDelta_\bGamma$-functions $f_n\colon \RR \to D_n$ and, for
every $z \in \Can$, a sequence $\{\alpha^m_k(z) \mid k,m \in \omega\}$
of ordinals strictly
smaller than
$\bdelta_\bGamma$ such that for every $n \in \omega$
\begin{equation}
  \label{eqcondDP2}
  \forall m \leq n \forall x,y \in D_{n+1} \forall k < n
  (\psi^{g_m}_k(g_{m+1} \circ \dotsc \circ g_n(x),g_m \circ \dotsc
  \circ g_n(x))= \alpha^m_k(z)),
\end{equation}
where $g_i = f_i$ if $z(i)=1$ and $g_i = id$ otherwise\footnote{The ordinals $\alpha^m_k(z)$ will really depend only
  on $\restr{z}{\max\{m,k\}+1}$, and we will also have that
  $\alpha^m_k(z) < \eta_{g_m}$ for every $m,k \in \omega$.} (when $m=n$,
the expression $\psi^{g_m}_k(g_{m+1}\circ \dotsc \circ g_n(x),g_m
\circ \dotsc \circ g_n(x)) = \alpha^m_k(z)$ in the equation above must be simply
understood as
$\psi^{g_n}_k (x,g_n(x)) = \alpha^n_k(z)$).

Having constructed all these sequences, we can finish the proof in the
following way: first fix $y_{n+1} \in D_{n+1}$ (for every $n \in
\omega$). For every $z \in \Can$, every $n \in \omega$ and every $m
\leq n$ define
$x^n_m = g_m \circ \dotsc \circ g_n(y_{n+1})$, 
and note that $x^n_m \in D_m$. If we fix $m$ and let vary the
parameter $n$ we get
\[ \lim_n \psi^{g_m}_k(x^n_{m+1},x^n_m) = \alpha^m_k(z) < \eta_{g_m} <
\bdelta_\bGamma \]
for every $k \in \omega$. This implies, by the property \eqref{verygood} of the scales involved, that the sequence $\seq{(x^n_{m+1},x^n_m)}{n \in \omega}$
converges
to some $(x_{m+1},x_m) \in {\rm graph}(g_m)$, that is to some pair of
points
such that $x_m \in D_m$ and $g_m(x_{m+1}) = x_m$. Observe also that
the sequence $\seq{x_m}{m \in \omega}$ is well defined since
$\seq{(x_n,y_n)}{n \in \omega}$ converges to $(x,y)$ if and only if
$\seq{x_n}{n \in \omega}$ converges to $x$ and $\seq{y_n}{n \in
  \omega}$ converges to $y$, and the limit of a converging sequence is unique.

Clearly, the points $x_m$ really depend on the choice of $z \in
\Can$, hence we should have written $x_m = x_m(z)$. If $z,w \in
\Can$ and $n_0 \in \omega$ are such that $\forall n > n_0 (z(n) =
w(n))$ then
$\forall n > n_0(x_n(z) = x_n(w))$,
and if $z(n_0) \neq w(n_0)$ then $x_{n_0}(z),x_{n_0}(w) \in D_{n_0}$ but
$x_{n_0}(z) \in A \cap D_{n_0} \iff x_{n_0}(w) \notin A \cap
D_{n_0}$.
Therefore we get that $\{ z \in \Can \mid x_0(z) \in A\}$ is a
flip-set, a contradiction!\\

Now we will construct by induction the $D_n$'s, the $f_n$'s and
the $\alpha^m_k$'s, granting inductively that $A$ is not decomposable in $D_n$. First put $D_0 = \RR$ and let $f_0$ be any
reduction of $A$ to $\neg A$.
Suppose to have constructed $D_j$, $f_j$ and
$\alpha^m_k(z)$ for every $j,m\leq n$, $k<n$ and $z \in
\Can$. Moreover fix $s \in {}^{n+1}2$ and define $g^s_i = g_i$ for
every $i \leq n$ by letting $g_i=f_i$ if $s(i) = 1$ and $g_i = id$
otherwise. For every $\tau \in {}^{n+1}(\eta_{g_0})$ consider the set
\[ D^0_\tau = \{ x \in D_n \mid \forall i < n+1 (\psi^{g_0}_i(g_1 \circ
\dotsc \circ g_n(x),g_0 \circ \dotsc \circ g_n(x)) =\tau(i))\}\]
(where if $n=0$ by $g_1 \circ \dotsc \circ g_n(x)$ we simply mean the
point $x$). Observe also that if $D^0_\tau \neq \emptyset$ then $\forall
j < n (\tau(j) = \alpha^0_i(z))$ for every $z \supseteq s$. Since $A$ is not decomposable in $D_n$
by inductive hypothesis, the fact
that  $\seq{D^0_\tau}{\tau \in
  {}^{n+1}(\eta_{g_0})}$ is a $\bDelta_\bGamma$-partition in less than
$\eta_{g_0} < \bdelta_\bGamma$ pieces of $D_n$ implies that there must be some
$\tau_0 \in {}^{n+1}(\eta_{g_0})$ such that $A$ is still not decomposable in $D^0_{\tau_0}$. 
Hence we can put $D^0 = D^0_{\tau_0} \subseteq D_n$, and
for every $z \in \Can$ such that $z \supseteq s$ and every $k < n+1$
we can also define
$\alpha^0_k(z) = \tau_0(k)$
(observe that since $A$ is not decomposable in $D^0_{\tau_0}$ then $D^0_{\tau_0}
\neq \emptyset$, and hence the definition of the $\alpha^0_k(z)$ is
well given).

Inductively, for every $m+1 < n+1$ we can repeat the above construction 
defining for every $\tau \in
{}^{n+1}(\eta_{g_{m+1}})$ the set
\[ D^m_\tau = \{ x \in D^m \mid \forall i < n+1 (\psi^{g_{m+1}}_i
(g_{m+2} \circ \dotsc \circ g_n(x),g_{m+1} \circ \dotsc \circ g_n(x))
= \tau(i))\}\]
(where if $m+1 = n$, as usual, $g_{m+2} \circ \dotsc \circ g_n(x)$
simply denotes the point $x$). The sequence $\seq{D^m_\tau}{\tau \in
  {}^{n+1}(\eta_{g_{m+1}})}$ forms a $\bDelta_\bGamma$-partition in
less than $\eta_{g_{m+1}}$ pieces of $D^m$, and since by inductive hypothesis $A$ is not decomposable in $D^m$  there must be some $\tau_{m+1}$
such that $A$ is still not decomposable in $D^m_{\tau_{m+1}}$. Moreover, for such
$\tau_{m+1}$ we have that $\alpha^{m+1}_k(z) = \tau_{m+1}(k)$ for every $k < n$ and every $s \subseteq z \in \Can$ (since $\emptyset \neq D^m_{\tau_{m+1}}
\subseteq D^m$).
Hence we can coherently define $D^{m+1} = D^m_{\tau_{m+1}}$ and
$\alpha^{m+1}_k(z) = \tau_{m+1}(k)$ for every $k < n+1$ and every $z
\in \Can$ such that $z \supseteq s$.

Now put $D(s) = D^n$ and repeat the whole construction for every
$s \in {}^{n+1}2$: let $\seq{s_i}{1 \leq i \leq 2^{n+1}}$ be an
enumeration without repetitions of ${}^{n+1}2$, and define $D(s_1)$
as above, $D(s_2)$ with the same construction but using $D(s_1)$
instead of $D_n$ in the first stage, and so on. Finally, put $D_{n+1} = D(s_{2^{n+1}})$, and let $f_{n+1} \in \F_\bGamma$ be obtained as
at the beginning of this proof. Clearly we have that $A$ is not decomposable in $D_{n+1}$, and it is straightforward to
inductively verify
that condition \eqref{eqcondDP2} holds for the sequences constructed.
\end{proof}

Now we want to prove the natural restatement of Lemma 4.5 of \cite{mottorosborelamenability} in this
new context\footnote{Note that Lemma 4.5 of \cite{mottorosborelamenability} is still true when
we consider superamenable sets of reductions $\F$, but if $\Bor \subsetneq \F$
this is not enough to determine the corresponding degree-structure.},
i.e.\ considering ${<\bdelta_\bGamma}$-partitions instead of countable
partitions.
The fundamental key to prove this result (and thus to determine the whole
degree-structure
induced by $\F_\bGamma$) is the following lemma, which unfortunately
(till the moment) can be proved only if $\bGamma$ is an existential
pointclass with the uniformization property.

\begin{lemma}[\AD]\label{lemmacentralrole}
  Let $\bGamma$ be an existential pointclass with the uniformization property.
For every $\eta<\bdelta_\bGamma$, every $\bDelta_\bGamma$-partition
$\seq{D_\xi}{\xi<\eta}$ of $\RR$, and every family of non trivial
(i.e.\ different from $\RR$) sets
$\{A_\xi
\mid
\xi<\eta\}$, we have that if $A_\xi \leq_\bGamma B$ for every $\xi<\eta$ then
$\bigcup_{\xi<\eta}(A_\xi \cap D_\xi) \leq_\bGamma B$ (for every $B \subseteq
\RR$).
\end{lemma}

\begin{proof}
  Since $A_\xi \cap D_\xi \leq_\bGamma A_\xi$, we can clearly assume
  that $A_\xi \subseteq D_\xi$ for every $\xi<\eta$ and prove that if
  $A_\xi
\leq_\bGamma B$ for every $\xi<\eta$ then $\bigcup_{\xi<\eta}A_\xi
\leq_\bGamma B$.

Let $G \subseteq \RR^3$ be a universal set for $\bGamma$, i.e.\ a set
in $\bGamma$ such that the sets $A \subseteq \RR^2$ which are in
$\bGamma$ are exactly those of the form $G_x= \{ (y,z) \in \RR^2 \mid (x,y,z) \in G\}$ for some $x \in \RR$.
For every $\xi < \eta$, let
\[ F_\xi = \{x \in \RR \mid G_x \text{ is the graph of a
  $\bDelta_\bGamma$-function which reduces }A_\xi \text{ to } B \},\]
and observe  that each $F_\xi$ is nonempty by our hypotheses. Let now
${\leq} \in \bDelta_\bGamma$ be a prewellordering of length $\eta$
(which exists since $\eta<\bdelta_\bGamma$), consider its strict part
$<$
(which is also in $\bDelta_\bGamma \subseteq \bGamma$) and let $\rho$
be its rank function (which is surjective on $\eta$). Now define
$f\colon  \eta \to \pow(\RR)\colon  \xi \to F_\xi$
and apply  Lemma \ref{codinglemma} to get
a choice set $C \in \bGamma$ for $f$, so that for every
$\xi<\eta$
there is some $(w,z) \in C$ such that $z \in F_{\rho(w)}$. Consider the relation (which is not necessarily the graph of a
function)
\[ \tilde{f} = \{ (x,y) \in \RR^2 \mid \exists w \exists z ((w,z) \in
C \wedge x \in D_{\rho(w)} \wedge (x,y) \in G_z)\}.\]
It is straightforward to check that $\tilde{f}$ is in $\bGamma$ and
hence admits a uniformization $f^*$ which is again in $\bGamma$. Thus
$f^*$ is the graph of a $\bGamma$-function $f$ (see
Proposition \ref{propdeltagammafunctions}), and we claim that $f$
reduces $\bigcup_{\xi<\eta}A_\xi$ to $B$. Fix some $x \in \RR$ and let
$\xi<\eta$ be (the unique ordinal) such that $x \in D_\xi$, so that $x
\in \bigcup_{\xi<\eta}A_\xi \iff x \in A_\xi$. Now we have that
$(x,f(x)) \in f^* \subseteq \tilde{f}$, and thus $(x,f(x))$ is in the
graph
of some $\bDelta_\bGamma$-function that was a reduction of $A_\xi$ to
$B$. Hence
\[ x \in \bigcup\nolimits_{\xi<\eta} A_\xi\iff x \in A_\xi \iff f(x) \in B\]
and we are done.
\end{proof}

Now 
observe that for every
$\eta < \bdelta_\bGamma$ there is a $\bDelta_\bGamma$-partition of
$\RR$ into $\eta$ many pieces. In fact, let $\leq$ be a
prewellordering in $\bDelta_\bGamma$ of length $\eta+1 <
\bdelta_\bGamma$ (such a preordering must exist by definition of
$\bdelta_\bGamma$). Let $\rho \colon  S \onto \eta+1$ be its rank function. If
$\eta = \mu+1$,  put
$D = \{x \in S \mid \rho(x) < \mu\}$
and check that $\seq{D_\xi}{\xi<\eta}$ is the partition required if we
define $D_\mu = \neg D$ and $D_\xi = \{ x \in D \mid \rho(x) = \xi\}$ for
 $\xi < \mu$. If instead $\eta$ is limit, put
$ D = \{ x \in S \mid \rho(x) < \eta\} $
and check that $\seq{D_\xi}{\xi<\eta}$ is again as required if we
define $D_0 = \neg D$ and $D_{1+\xi} = \{ x \in D \mid \rho(x) = \xi\}$
for every $\xi<\eta$.

\begin{theorem}[\AD]\label{theorstructure2}
  Let $\bGamma$ be a tractable pointclass.  Then we have that:
  \begin{enumerate}[i)]
  \item if $[A]_\bGamma$ is limit of cofinality strictly less than
$\bdelta_\bGamma$
then $A \leq_\bGamma \neg A$;
\item for every $\eta<\bdelta_\bGamma$, every $\bDelta_\bGamma$-partition
$\seq{D_\xi}{\xi<\eta}$ of $\RR$ and every $A,B \subseteq \RR$, if $A
\cap
D_\xi
\leq_\bGamma B$ for every $\xi<\eta$ then $A \leq_\bGamma B$;
\item if $[A]_\bGamma$ is limit of cofinality greater then
  $\bdelta_\bGamma$
then $A
\nleq_\bGamma \neg A$;
\item after an $\F_\bGamma$-seldual degree there is a nonselfdual pair.
  \end{enumerate}
\end{theorem}

\begin{proof}
For part \textit{i)}, let ${\rm cof}_\bGamma(A) = \eta < \bdelta_\bGamma$ and let
  $\{[A_\xi]_\bGamma \mid \xi < \eta \}$ be any family of degrees such
  that $A_\xi <_\bGamma A$ for every $\xi<\eta$ and such that for
  every $B$ for which $\forall \xi<\eta(A_\xi \leq_\bGamma B)$ we have
  that $B \nless_\bGamma A$.
  Let $\seq{D'_\xi}{\xi<\eta}$ be any $\bDelta_\bGamma$-partition of
  $\RR$ (which must exists by the observation preceding this
  theorem), and for $\xi < \eta$ define
$D_\xi = \{ x \oplus y \in \RR \mid y \in D'_\xi \}$,
$C_\xi = \{ x \oplus y \in D_\xi \mid x \in A_\xi\}$, and $C = \bigcup_{\xi < \eta} C_\xi$. Note that $C_\xi \subseteq D_\xi$ and $C_\xi \equiv_\bGamma A_\xi$ for every $\xi < \eta$.
It is clear that we can
assume $C_\xi \neq \RR$ for every $\xi<\eta$ and apply Lemma
\ref{lemmacentralrole} to the $\bDelta_\bGamma$-partition
$\seq{D_\xi}{\xi<\eta}$ and to the $C_\xi$'s to get $C \leq_\bGamma
A$. Conversely, for every $\xi<\eta$
\[ A_\xi \equiv_\bGamma C_\xi = C \cap D_\xi \leq_\bGamma C,\]
hence $A \leq_\bGamma C$ by our hypotheses (since otherwise $C
<_\bGamma A$). Thus it is enough to show that $C$ is
$\bGamma$-selfdual. To see this, observe that since $C_\xi
\equiv_\bGamma A_\xi <_\bGamma A \equiv_\bGamma C$ we have also $C_\xi
<_\bGamma \neg C$ by $\SLO^\bGamma$ (which follows from $\AD$): therefore we can apply Lemma \ref{lemmacentralrole}
again with $B = \neg C$ to get $C \leq_\bGamma \neg C$.

For part \textit{ii)} simply apply Lemma \ref{lemmacentralrole} with $A_\xi = A \cap
  D_\xi$ (for every $\xi<\eta$).

For part \textit{iii)}, assume that $[A]_\bGamma$ is limit (in particular, $A \notin
\Delta_{\F_\bGamma}$) and $A \leq_\bGamma \neg
  A$. By Theorem \ref{theorDP2} there is some $\eta<\bdelta_\bGamma$
  and a $\bDelta_\bGamma$-partition $\seq{D_\xi}{\xi<\eta}$ of $\RR$
  such that $A \cap D_\xi <_\bGamma A$ for every $\xi<\eta$. By part
  \textit{ii)}, $[A]_\bGamma$ is the supremum of the family
  $\mathcal{A} = \{ [A
  \cap D_\xi]_\bGamma \mid \xi<\eta\}$, and hence $\mathcal{A}$
  witnesses that $[A]_\bGamma$ is cofinality stricly less than
  $\bdelta_\bGamma$. Therefore, if $[A]_\bGamma$ is limit of
  cofinality greater than $\bdelta_\bGamma$ then $A \nleq_\bGamma \neg
  A$.

Finally, for part \textit{iv)} it is enough to prove that if $A$ and $B$ are two
  $\F_\bGamma$-selfdual
sets such that $A <_\bGamma B$ (which implies $B \notin
\Delta_{\F_\bGamma}$), then there is some $C$ such that $A
<_\bGamma C <_\bGamma B$. By Theorem \ref{theorDP2} again, there must
be some $\eta<\bdelta_\bGamma$ and some $\bDelta_\bGamma$-partition
$\seq{D_\xi}{\xi<\eta}$ of $\RR$ such that $B \cap D_\xi <_\bGamma B$
for every $\xi<\eta$. If $B \cap D_\xi \leq_\bGamma A$ for every $\xi
<\eta$, then we would have $B \leq_\bGamma A$ by part \textit{ii)}, a
contradiction! Hence there must be some $\xi_0 <\eta$ such that $B
\cap D_{\xi_0} \nleq_\bGamma A$, and by $\SLO^\bGamma$ and
$\F_\bGamma$-selfduality of $A$, we get $A<_\bGamma B \cap D_{\xi_0}
<_\bGamma B$.
\end{proof}

The previous theorem shows that if $\bGamma$ is tractable we can completely describe the hierarchy of the $\F_\bGamma$-degrees using Theorem \ref{theorstructure2} and
Theorem 3.1 of \cite{mottorosborelamenability}:
it is a well-founded preorder of length $\Theta$, nonselfdual pairs and selfdual
degrees alternate, at limit levels of cofinality strictly less than
$\bdelta_\bGamma$ there is a selfdual degree, while at limit levels of
cofinality equal or greater than $\bdelta_\bGamma$ there is a nonselfdual pair.
Thus the degree-structure infuced by $\F_\bGamma$ looks like this:

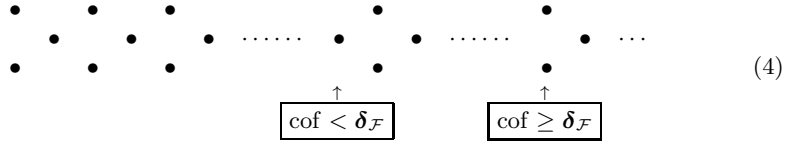
\begin{small}
\begin{equation}\label{picture}
\begin{array}{llllllllllllll}
\bullet & & \bullet & & \bullet & & & & \bullet & & & \bullet
\\
& \bullet & & \bullet & & \bullet & \cdots \cdots & \bullet
& & \bullet &\cdots\cdots & & \bullet & \cdots
\\
\bullet & & \bullet & & \bullet & & & & \bullet & & & \bullet
\\
& & & & & & &
\, \makebox[0pt]{$\stackrel{\uparrow}{\framebox{${\rm cof} <
\bdelta_\F$}}$}
& & & &
\, \makebox[0pt]{$\stackrel{\uparrow}{\framebox{${\rm cof}
\geq \bdelta_\F$}}$}
\end{array}
\end{equation}
\end{small}

Note that the previous picture is coherent with the description of the structure of the $\F$-degrees when $\F$ is Borel-amenable: in fact, as already observed, in that case we have $\bdelta_\F = \omega_1$, and therefore  picture \eqref{picture} coincides with the usual one.

\end{document}